\documentclass[12pt,reqno]{amsart}

\usepackage[shortlabels]{enumitem}
\usepackage[hidelinks]{hyperref}
\usepackage[capitalise]{cleveref}
\usepackage{amssymb}
\usepackage{amsthm}
\usepackage{amsmath}
\usepackage{a4wide}
\usepackage{latexsym}
\usepackage{mathrsfs}
\usepackage[T1]{fontenc}
\usepackage[latin1]{inputenc}

\numberwithin{equation}{section}

\newtheorem{theorem}{Theorem}[section]
\newtheorem{lemma}[theorem]{Lemma}

\newtheorem{corollary}[theorem]{Corollary}
\theoremstyle{definition}
\newtheorem{definition}[theorem]{Definition}
\newtheorem{question}[theorem]{Question}

\newtheorem{remark}[theorem]{Remark}

\renewcommand{\le}{\ensuremath{\leqslant}}

\renewcommand{\ge}{\ensuremath{\geqslant}}

\newcommand{\N}{\mathbb{N}}
\newcommand{\R}{\mathbb{R}}
\newcommand{\C}{\mathbb{C}}
\newcommand{\K}{\mathbb{K}}

\title[Kernels of operators on classical transfinite sequence spaces]{Kernels of bounded operators on the classical transfinite Banach sequence spaces}

\author[M.~Arnott]{Max Arnott}
\address[M.~Arnott]{Zaiku Group ltd, 7th Floor, 4~St Paul's Square, Liverpool, L3~9SJ, United Kingdom}
\email{arnott.max@zaikugroup.com}

\author[N.~J.~Laustsen]{Niels Jakob Laustsen}
\address[N.~J.~Laustsen]{School of Mathematical Sciences, Fylde College, Lancaster University, Lancaster, LA1 4YF, United Kingdom}
\email{n.laustsen@lancaster.ac.uk}



\begin{document}

\begin{abstract}
Every closed subspace of each of the Banach spaces $X = \ell_p(\Gamma)$ and  \mbox{$X=c_0(\Gamma)$}, where~$\Gamma$ is a set and $1<p<\infty$, is the kernel of a bounded operator $X\to X$.
    On the other hand, whenever~$\Gamma$ is an uncountable set, $\ell_1(\Gamma)$ contains a closed subspace that is not the kernel of any bounded operator $\ell_1(\Gamma)\to\ell_1(\Gamma)$.
\end{abstract}

\maketitle

\section{Introduction and main results}
\noindent This paper can be seen as a continuation of the line of research initiated by Laust\-sen and White  in~\cite{LW}. More precisely, our aim is to answer the following question for the classical transfinite Banach sequence spaces~$c_0(\Gamma)$ and~$\ell_p(\Gamma)$ for $1\le p<\infty$.

\begin{question}[The kernel problem]\label{kernelproblem}
Let $X$ be a Banach space. Is every closed subspace of~$X$ equal to the kernel of some bounded operator $X\to X$?
\end{question}

The main result of~\cite{LW} is that there exists a reflexive Banach space~$X$ which answers \Cref{kernelproblem} in the negative; that is, $X$~contains a closed subspace which is not the kernel of any bounded operator $X\to X$.  We refer to~\cite{LW}, as well as~\cite{White}, for the motivation behind this result, including the significance of~$X$ being reflexive. 

The space~$X$ that was used to prove this result is the dual of an ``exotic'' non-sep\-a\-ra\-ble Banach space constructed by Wark~\cite{wark1} with the property that every bound\-ed operator on it is the sum of a scalar multiple of the identity operator and an operator with separable image. It turns out that Wark's Banach space (not its dual) can also be used, as we shall explain in \Cref{S:answerforWark}.

The fact that~$X$ is non-separable is necessary; this follows from \cite[Proposition~2.1]{LW}. 
However, it is not clear that Wark's highly sophisticated Banach space is needed to obtain such an example. This raises the question: is it possible to find a simpler example of a reflexive Banach space~$X$ which contains a closed subspace that is not the kernel of any bound\-ed operator $X\to X$? Obviously~$X$ cannot be a Hilbert space, else its closed subspaces would all be complemented. Our main result is that one cannot use any of the other classical reflexive transfinite sequence spaces either. 

\begin{theorem}\label{T:main1}
    Let~$\Gamma$ be a set, and take $1<p<\infty$. Then, for every closed subspace~$Y$ of~$\ell_p(\Gamma)$, there is a bounded operator $T\colon \ell_p(\Gamma)\to\ell_p(\Gamma)$ with $\ker T = Y$.  
\end{theorem}

Although the original application in~\cite{White} required that the Banach space~$X$ is reflexive, \Cref{kernelproblem} makes sense --- and is equally natural --- in the non-reflexive setting. In fact,  Kalton already answered it for $X=\ell_\infty$ 50 years ago:  building on Whitley's simplified proof~\cite{Whit} of Phillips' theorem, Kalton \cite[Proposition~4]{Kalton} showed that whenever the kernel of a bounded operator $T\colon\ell_\infty\to\ell_\infty$ contains~$c_0$, we can find an infinite subset~$M$ of~$\N$ such that $\{x\in\ell_\infty : x(n)=0\ (n\in\N\setminus M)\}\subseteq \ker T$. 
Thus in particular, $c_0$ is not the kernel of any bounded operator on~$\ell_\infty$. 
Related results,  including a generalisation of Kalton's theorem to~$\ell_\infty(\Gamma)$ for uncountable index sets~$\Gamma$, have recently been obtained in~\cite{HL}. 

Complementing \Cref{T:main1}, we provide the following answers to \Cref{kernelproblem} for the non-reflexive classical transfinite sequence spaces.

\begin{theorem}\label{T:main2}
  Let~$\Gamma$ be a set. 
  \begin{enumerate}[label={\normalfont{(\roman*)}}]
  \item\label{T:main2i} For every closed subspace~$Y$ of~$c_0(\Gamma)$, there is a bounded operator $T\colon c_0(\Gamma)\to c_0(\Gamma)$ with $\ker T = Y$. 
  \item\label{T:main2ii} Suppose that~$\Gamma$ is uncountable. Then~$\ell_1(\Gamma)$ contains a closed subspace which is not the kernel of any bounded operator $\ell_1(\Gamma)\to\ell_1(\Gamma)$.
  \end{enumerate}
\end{theorem}

This result shows in particular that a Banach space~$X$ --- namely $X=c_0(\Gamma)$ for~$\Gamma$ uncountable --- may answer \Cref{kernelproblem} positively, whilst its dual $X^* =\ell_1(\Gamma)$ answers it negatively. We note that the answer for the bidual $X^{**} = \ell_\infty(\Gamma)$ is also negative by~\cite{HL}. Wark's Banach space and its dual remain the only reflexive spaces~$X$ known to contain a closed subspace that is not the kernel of any bounded operator $X\to X$. These observations motivate the following questions.

\begin{question} 
\begin{enumerate}[label={\normalfont{(\roman*)}}]
\item\label{Q:oppositeanswers}   Does there exist a Banach space $X$ for which the answer to \Cref{kernelproblem} is negative, whilst $X^*$ answers \Cref{kernelproblem} positively?
\item  In particular, is it possible to find a reflexive Banach space $X$ which has the properties described in ~\ref{Q:oppositeanswers}? 
\end{enumerate}
\end{question}

\begin{remark}
  \Cref{kernelproblem} asks whether every closed subspace~$Y$ of~$X$ must be the kernel of some bounded operator $X\to X$, and not whether~$Y$ is the kernel of a bounded operator from~$X$ into some arbitrary Banach space. This is because the quotient map $X \to X/Y$ has kernel $Y$, so the question would be trivial in this generality. 
\end{remark}

\section{Preliminaries and the proof of Theorem \ref{T:main1}}

\noindent Let $\K = \R$ or $\K = \C$ denote the scalar field. For a set~$\Gamma$, we consider the Banach spaces
\[ \ell_p(\Gamma) = \biggl\{ x\in \K^\Gamma : \sum_{\gamma\in\Gamma} \lvert x(\gamma)\rvert^p < \infty\biggr\}\qquad (1\le p<\infty) \]
and 
\[ c_0(\Gamma) = \bigl\{ x\in \K^\Gamma : \{\gamma\in\Gamma : \lvert x(\gamma)\rvert\ge\varepsilon\}\ \text{is finite for every}\ \varepsilon>0 \bigr\}. \]
For $x\in\K^\Gamma$, $\operatorname{supp} x = \{ \gamma\in\Gamma : x(\gamma)\ne 0\}$ denotes the \emph{support} of~$x$.
We shall frequently use that every element of each of the spaces  $X= \ell_p(\Gamma)$ for $1\le p<\infty$  and  $X=c_0(\Gamma)$ has countable support, a fact easily derived from the definition of~$X$. We write $(e_\gamma)_{\gamma\in\Gamma}$ for the \emph{unit vector basis} of~$X$; that is, $e_\gamma(\beta) = 1$ if $\beta = \gamma$ and $e_\gamma(\beta) =0$ otherwise. This defines a symmetric transfinite Schauder basis for~$X$, for which we let $(e_\gamma^*)_{\gamma\in\Gamma}$ denote the set of associated coordinate functionals.
 
We begin with a lemma which outlines how duality can be brought to bear on \Cref{kernelproblem} for reflexive Banach spaces. Similar ideas were used in~\cite{LW}. 

\begin{lemma}\label{kernelreflexive} 
  Let $X$ be a reflexive Banach space. Every closed subspace of~$X$ is the kernel of some bounded operator $X\to X$ if and only if every closed subspace of~$X^*$ is the closure of the image of some bounded operator $X^*\to X^*$.
\end{lemma}

The proof involves the following standard notions.
The \emph{annihilator} of a subset~$Y$ of a Banach space~$X$ is 
\[ Y^\perp = \{ f \in X^* : \langle y, f\rangle = 0\ \text{for every}\ y\in Y \}, \] 
and the \emph{pre-annihilator} of a subset~$Z$ of a dual Banach space~$X^*$ is 
\[ Z_\perp = \{ x \in X : \langle x, f\rangle = 0\ \text{for every}\ f\in Z \}. \] 
Note that these definitions depend on $X$, as well as~$Y$ or~$Z$. Simple direct calculations show that~$Y^\perp$ is a weak*-closed subspace of~$X^*$, $Z_\perp$  is a norm-closed subspace of~$X$, $(Y^\perp)_\perp = \overline{\operatorname{span}}\, Y$ and $(Z_\perp)^\perp$ is the weak*-closure of $\operatorname{span} Z$. Furthermore, we shall use the well-known identity 
$\ker T = T^*[X^*]_\perp$ for any bounded operator $T\colon X\to X$. 
 
 \begin{proof}[Proof of Lemma~{\normalfont{\ref{kernelreflexive}}}]
    Suppose that every closed subspace of~$X$ is the kernel of some bound\-ed operator $X\to X$. Then, given a closed subspace~$Z$ of $X^*$, we can take a bounded operator $T\colon X\to X$ with $\ker T=Z_\perp$.    
    By reflexivity, closed subspaces are automatically weak*-closed, so 
    \[ Z = (Z_{\perp})^\perp  = (\ker T)^\perp = (T^*[X^*]_\perp)^\perp = \overline{T^*[X^*]}. \]

    On the other hand, suppose that every closed subspace of $X^*$ is the closure of the image of some bounded operator $X^*\to X^*$. Then, given a closed subspace~$Y$ of~$X$, we can take a bounded operator $S\colon X^*\to X^*$ such that $Y^\perp= \overline{S[X^*]}$. Reflexivity implies that $S = T^*$ for some bounded operator $T\colon X\to X$, and therefore  
    \[ \ker T = T^*[X^*]_\perp = S[X^*]_\perp = \overline{S[X^*]}_\perp = (Y^\perp)_\perp = Y. \qedhere   \] \end{proof}

Solving the kernel problem positively for $\ell_p(\Gamma)$ for $1<p<\infty$ therefore equates to proving that every closed subspace of $\ell_{p^*}(\Gamma)$ is the closure of the image of some bounded operator $\ell_{p^*}(\Gamma)\to\ell_{p^*}(\Gamma)$, where $p^*\in(1,\infty)$ denotes the conjugate exponent of~$p$. The similarity of structure between reflexive transfinite $\ell_p$-spaces and their duals means that we can shift our attention to subspaces which are closures of images of bounded operators rather than  kernels thereof. The following result characterises such subspaces in terms of the supports of elements in their spanning subsets.

\begin{lemma}\label{Prop2.5}
Let~$Y$ be a closed subspace of~$\ell_p(\Gamma)$ for some $1<p<\infty$ and some un\-count\-able set~$\Gamma$. The following conditions are equivalent:
\begin{enumerate}[label={\normalfont{(\alph*)}}]
    \item\label{Prop2.5a} There exists a bounded operator $\ell_p(\Gamma)\to Y$ with dense image.
    \item\label{Prop2.5d} $Y$ contains a subset $D$ for which $\overline{\operatorname{span}}\, D=Y$ and the set $\{d\in D : d(\gamma)\ne 0\}$ is countable for every $\gamma\in\Gamma$.
    \item\label{Prop2.5c} $Y$ contains a sequence of subsets $(D_n)_{n\in\mathbb{N}}$ for which $\overline{\operatorname{span}}\,\bigcup_{n\in\mathbb{N}} D_n=Y$ and 
    \begin{equation}\label{Prop2.5c:eq} \operatorname{supp}x\cap\operatorname{supp}y=\emptyset\qquad (n\in\mathbb{N},\, x,y \in D_n, \, x \neq y). \end{equation}
\end{enumerate}
\end{lemma}

\begin{proof}  \ref{Prop2.5a}$\Rightarrow$\ref{Prop2.5d}.
  Suppose that $T\colon\ell_p(\Gamma)\to Y$ is a bounded operator with $\overline{T[\ell_p(\Gamma)]} = Y$. We shall show that  the set $D = \{Te_\beta :\beta\in \Gamma\}$ satisfies the conditions of~\ref{Prop2.5d}. First, $\overline{\operatorname{span}}\, D=Y$ because $\{e_\beta : \beta\in \Gamma\}$ spans a dense subspace of~$\ell_p(\Gamma)$. Second, let $J\colon Y\to\ell_p(\Gamma)$ denote the inclusion map, and take $\gamma\in\Gamma$. The calculation  
  \[ (Te_\beta)(\gamma) = \langle  JTe_\beta, e_\gamma^*\rangle = \langle  e_\beta,T^*J^*e_\gamma^*\rangle = (T^*J^*e_\gamma^*)(\beta)\qquad (\beta\in\Gamma) \]
  implies that
  \[ \{d\in D : d(\gamma)\ne 0\} = 
  \{ Te_\beta : \beta\in\Gamma,\, (Te_\beta)(\gamma)\ne 0\}  = \{ Te_\beta : \beta\in\operatorname{supp}T^*J^*e_\gamma^*\}, \]
  and this set is countable because $T^*J^*e_\gamma^*\in\ell_{p^*}(\Gamma)$ has countable support, where~$p^*$ denotes the conjgugate exponent of~$p$, as before.

  \ref{Prop2.5d}$\Rightarrow$\ref{Prop2.5c}. Suppose that $D\subseteq Y$ satisfies the conditions of \ref{Prop2.5d}. Replacing $D$ with $D\setminus\{0\}$,  we may suppose that $0 \notin D$.  Define the maps
  \begin{alignat*}{2} \varphi\colon\ &\mathscr{P}(D)\to \mathscr{P}(D),\quad && B\mapsto\{d \in D :  \operatorname{supp}d\cap\operatorname{supp}b\ne \emptyset\ \text{for some}\ b\in B\},\\ 
  \psi\colon\ &D\to \mathscr{P}(D),\quad && d\mapsto \bigcup_{n \in \mathbb{N}}\varphi^n(\{d\}),
  \end{alignat*} where $\mathscr{P}(D)$ denotes the power set of $D$. 
  By hypothesis, the set $\{d \in D : d(\gamma)\neq 0\}$ is countable for every $\gamma\in\Gamma$. Combining this with the identity
  \[ \varphi(B) = \bigcup_{b\in B} \bigcup_{\gamma\in\operatorname{supp} b}\{d \in D : d(\gamma)\neq 0\} \]
  and the fact that vectors in~$\ell_p(\Gamma)$ have  countable support, we deduce that~$\varphi(B)$ is countable for every countable subset~$B$ of~$D$. By induction, it follows that $\varphi^n(B)$ is countable for every $n \in \mathbb{N}$ and $B\subseteq D$ countable. In particular, $\psi(d)$ is countable for every $d \in D$, being a countable union of countable sets.

  Define a relation~$\sim$ on~$D$ by $c\sim d$ if and only if there are $n\in\N$ and elements $b_1,\dots, b_{n+1}\in D$ such that $c=b_1$, $d=b_{n+1}$ and $\operatorname{supp} b_j\cap\operatorname{supp} b_{j+1}\ne\emptyset$ for each $j\in\{1,\dots,n\}$. It is easy to check that $\sim$ is an equivalence relation on $D$. (Re\-flex\-iv\-i\-ty relies on the fact that $0\notin D$.) 
  Let $[d] = \{ c\in D : c\sim d\}$ denote the equivalence class of $d\in D$, and observe that $[d] = \psi(d)$, so it is countable. 
  
  Choose a subset~$C$ of~$D$ such that $D/{\sim} = \{[c] : c\in C\}$ and $[c]\ne [d]$ whenever $c,d\in C$ are distinct, and enumerate each of these equivalence classes as 
  \[ [c] = \{d^c_n : n\in\mathbb{N}\}\qquad (c\in C). \]
  Two elements of~$D$ which belong to different equivalence classes must have disjoint support by the definition of~$\sim$, therefore the sequence of sets $(D_n)$ defined by $D_n=\{d^c_n : c\in C\}$ for $n \in \mathbb{N}$ satisfies the conditions of~\ref{Prop2.5c}.   
  
  \ref{Prop2.5c}$\Rightarrow$\ref{Prop2.5a}. Let $(D_n)_{n\in\mathbb{N}}$ be a sequence of subsets of~$Y$ which satisfies the conditions of~\ref{Prop2.5c}, and take $n\in\N$. We may suppose that every element of~$D_n$ has norm~$1$ by removing~$0$ and replacing the non-zero elements with their normalisations. Then~\eqref{Prop2.5c:eq} implies that the elements of~$D_n$ span an isometric copy of~$\ell_p(D_n)$ inside~$Y$; that is, we can define a linear isometry $U_n\colon \ell_p(D_n)\to Y$ by $U_n(e_d) = d$ for every $d\in D_n$, where $(e_d)_{d\in D_n}$ denotes the unit vector basis of~$\ell_p(D_n)$. (To avoid any possible confusion, let us stress that this is \emph{not} a vector-valued $\ell_p$-space; $D_n$~is sim\-ply used as an index set here.)  

  It follows from~\eqref{Prop2.5c:eq} that~$D_n$ has cardinality no greater than~$\Gamma$, and therefore the same is true for the set $A = \bigcup_{n\in\N} D_n\times\{n\}$. Take an injection $\theta\colon A\to\Gamma$, and define $\theta_n\colon D_n\to\Gamma$ by $\theta_n(d) = \theta(d,n)$ for each $n\in\N$. Then $x\mapsto x\circ\theta_n$ defines a bounded operator $\ell_p(\Gamma)\to\ell_p(D_n)$ of norm~$1$, so we can define a bounded operator by
  \[ T\colon\ \ell_p(\Gamma)\to Y,\quad x\mapsto \sum_{n=1}^\infty \frac{U_n(x\circ\theta_n)}{2^n}. \] 
  Take $n\in\N$ and $d\in D_n$. Straightforward calculations show that $e_{\theta(d,n)}\circ\theta_n = e_d$ and $e_{\theta(d,n)}\circ\theta_m = 0$ for $m\in\N\setminus \{ n\}$, so $T(2^ne_{\theta(d,n)}) = U_n(e_d) = d$. It follows that $D_n\subseteq T[\ell_p(\Gamma)]$, which implies that the image of~$T$ is dense in~$Y$. 
\end{proof}  

\begin{remark}
We include here an alternative proof of the implication \ref{Prop2.5d}$\Rightarrow$\ref{Prop2.5c} in the above theorem using graph-theoretic methods, in the hope that the reader might find the method novel or illuminating in some respect. 

The proof involves the following standard terminology from graph theory:
\begin{itemize}
   \item A \textit{graph} is an ordered pair $G=(V,E)$, where $V$ is the set of \textit{vertices} of $G$, and $E$ is a set of unordered pairs of distinct vertices called the set of \textit{edges} of~$G$. (Note that this definition of edges specifically excludes loops.)   
   \item A \textit{subgraph} of $G$ is a graph $G'= (V',E')$ for which $V'\subseteq V$ and $E'\subseteq E$. 
   \item The \textit{degree} of a vertex $v\in V$ is the cardinality of the set $\{e \in E : v \in e\}$.

   \item A \textit{colouring} of $G$ is a map $f\colon V \to C$ for some set $C$ of \textit{colours}. A colouring of $G$ is \textit{proper} if $f(x)\neq f(y)$ whenever $\{x,y\}\in E$.

   \item A \textit{path} in $G$ beginning at a vertex $v_1\in V$ and ending at $v_n \in V$ is a finite sequence of vertices $(v_1,v_2,\dots,v_n)$ for some $n \in \mathbb{N}$ such that $\{v_j,v_{j+1}\}\in E$ for every \mbox{$j \in \{1,\dots,n-1\}$}. Two vertices in $V$ are \textit{connected} if they belong to a common path, and a graph is \textit{connected} if any two vertices within it are connected.

  \item  A \textit{connected component} of $G$ is a connected subgraph of $G$ which is not a subgraph of any strictly larger connected subgraph of $G$.
    \end{itemize}

\begin{proof}[Alternative proof of \ref{Prop2.5d}$\Rightarrow$\ref{Prop2.5c}] Take a subset~$D$ of~$Y$ which satisfies the conditions of~\ref{Prop2.5d}, and define the graph $G=(V,E)$ with vertices $V=D$ and edges \[E=\{\{b,d\} \subseteq V : b \neq d\ \text{and}\  \operatorname{supp}b\cap \operatorname{supp}d\neq \emptyset \}. \] 
Each vertex $d\in V$ has countable support, and by hypothesis there are only countably many vertices in~$V$ that are supported at any given coordinate $\gamma\in\Gamma$. Hence~$d$ has degree at most~$\aleph_0$.
It follows that for each $n\in\N$, there are at most $\aleph_0^n$ many vectors connected to $d$ via a path of length~$n$, so there are at most $\bigl\lvert\bigcup_{n\in\N}\aleph_0^n\bigr\rvert = \aleph_0$ many vertices which are contained in a path beginning at~$d$. This means that the number of vertices in the connected component of~$G$ containing~$d$ is countable.

We can therefore define a proper colouring $f\colon V \to \mathbb{N}$ of $G$ with $\aleph_0$ many colours, simply using each colour at most once per connected component. Then the sets 
 \[ D_n = \{v \in V : f(v)=n\}\qquad (n \in \mathbb{N}) \] 
 partition $D$ and satisfy~\eqref{Prop2.5c:eq}.
\end{proof}
\end{remark}

We conclude this section by proving that every closed subspace of $\ell_p(\Gamma)$ contains a spanning set of the form described in \cref{Prop2.5}\ref{Prop2.5d}, thus solving the kernel problem in this setting. 

\begin{definition} A \emph{Markushevich basis} for a Banach space~$Y$ is a family $(y_j,f_j)_{j\in J}$ indexed by some set~$J$, where $y_j\in Y$ and $f_j\in Y^*$ for each $j\in J$, and the following three conditions are satisfied: 
\begin{equation}\label{D:Markushevich:eq}
  \langle y_j, f_k\rangle = \delta_{j,k}\quad (j,k\in J),\qquad \overline{\operatorname{span}}\,\{y_j : j\in J\} = Y\quad\text{and}\quad \bigcap_{j\in J}\ker f_j = \{0\}. 
\end{equation}
If the functionals $\{f_j:j\in J\}$ span a norm-dense subspace of~$Y^*$, then we say that the Mar\-ku\-she\-vich basis $(y_j,f_j)_{j\in J}$ is  \emph{shrinking}. 
\end{definition}

\begin{lemma}\label{KeyLemmaA}
  Let~$Y$ be a Banach space with a shrinking Markushevich basis $(y_j,f_j)_{j\in J}$. Then the set $\{ j\in J : \langle   y_j,f\rangle\ne 0\}$ is countable for every $f\in Y^*$.
  \end{lemma}

\begin{proof}
Clearly the set   $\{ j\in J : \langle y_j,g\rangle\ne 0\}$ is finite whenever $g\in\operatorname{span} \{f_j:j\in J\}$. Since the  Markushevich  basis is shrinking, every $f\in Y^*$ is the norm limit of a sequence $(g_n)_{n\in\N}$ in $\operatorname{span} \{f_j:j\in J\}$, and therefore
\[  \{ j\in J : \langle y_j,f\rangle\ne 0\}\subseteq\bigcup_{n\in\N}\{ j\in J : \langle y_j,g_n\rangle\ne 0\} \]
is countable.
\end{proof}

\begin{remark} \Cref{KeyLemmaA} is a special case of a more general characterisation of weakly Lindel\"{o}f determined Banach spaces \cite[Theorem~5.37]{HMVZ}. We have included the short, elementary argument above to keep the presentation as self-contained as possible. 
  \end{remark}

\begin{corollary}\label{KeyLemmaB}
Every reflexive Banach space~$Y$ contains a subset~$D$ such that $\overline{\operatorname{span}}\, D = Y$ and the set $\{ d\in D : \langle d,f\rangle\ne 0\}$ is countable for every $f\in Y^*$. 
\end{corollary}

\begin{proof}
  Being reflexive, $Y$ admits a shrinking Markushevich  basis $(y_j,f_j)_{j\in J}$ (see, \emph{e.g.,} \cite[Theorem~6.1]{HMVZ}). Set $D = \{ y_j : j\in J\}$. Then $\overline{\operatorname{span}}\, D = Y$ by the second part of the definition~\eqref{D:Markushevich:eq}, while \Cref{KeyLemmaA} implies that the set $\{ d\in D : \langle d,f\rangle\ne 0\}$ is countable for every $f\in Y^*$. 
\end{proof} 

\begin{proof}[Proof of Theorem {\normalfont{\ref{T:main1}}}]
  Take a closed subspace~$Y$ of~$\ell_{p^*}(\Gamma)$, where 
  $p^*\in(1,\infty)$ denotes the conjugate exponent of~$p$. Then~$Y$ is reflexive, so \Cref{KeyLemmaB} shows that it 
  contains a subset~$D$ such that $\overline{\operatorname{span}}\, D = Y$ and the set $\{ d\in D : \langle d, f\rangle\ne 0\}$ is countable for every $f\in Y^*$. In particular, choosing $f = e_\gamma^*|_Y$ for $\gamma\in\Gamma$, we see that~$D$ satisfies the  conditions of \Cref{Prop2.5}\ref{Prop2.5d}. Therefore, by \Cref{Prop2.5}\ref{Prop2.5a}, we can take a bounded operator $\ell_{p^*}(\Gamma)\to Y$ with dense image. Now the conclusion follows from \Cref{kernelreflexive}. 
\end{proof}

\begin{remark}
    After we completed the research presented in this section and communicated our findings in a webinar, Bill Johnson informed us that it is known that every closed subspace~$Y$ of~$\ell_p(\Gamma)$ for $1<p<\infty$ is isometrically isomorphic to the $\ell_p$-sum of a family of separable Banach spaces. With this result at hand, it is easy to construct a bounded operator $\ell_p(\Gamma)\to Y$ with dense image and consequently deduce \Cref{T:main1} from \Cref{kernelreflexive}, as above.     
    
    Johnson referred us to \cite[Section~4]{JohnsonOdell} for said result, where it can be found within the proof of \cite[Theorem~3]{JohnsonOdell}, whose statement concerns the structure of subspaces of~$L_p(\mu)$ for $2<p<\infty$ in the non-separable case.     
    Inspection of this proof reveals not only that the parts which are needed to express $Y\subseteq\ell_p(\Gamma)$ as the $\ell_p$-sum of a family $(Y_\xi)_{\xi\in\Xi}$  of separable Banach spaces work equally well for $1<p<2$, but also that our arguments are very similar to those Johnson and Odell employed 50 years ago. 
    
    We feel that this result deserves wider publicity, so let us outline its proof. 
    Like us, Johnson and Odell begin by taking a  Markushevich  basis $(y_j,f_j)_{j\in J}$ for~$Y$, which they normalise, so that $\lVert y_j\rVert = 1$ for each $j\in J$, and they then justify that the set \mbox{$\{ j\in J : y_j(\gamma)\ne0\}$} is countable for each $\gamma\in\Gamma$, as we did in the proof of \Cref{KeyLemmaB}/\Cref{T:main1}. Their next step is to recursively define a partition of the index set~$J$ into a disjoint union $J=\bigcup_{\xi\in\Xi}K_\xi$, for some index set~$\Xi$, where~$K_\xi$ is countable for each $\xi\in\Xi$ and 
    \begin{equation}\label{Eq:WBJEO}
        \operatorname{supp} y_j\cap \operatorname{supp} y_k=\emptyset\qquad (j\in K_\xi,\,k\in K_\eta,\, \xi,\eta\in\Xi,\,\xi\ne\eta).
    \end{equation}
    Bearing in mind that $D = \{ y_j : j\in J\}$ in our proof \Cref{T:main1}, we see that this partition corresponds to our collection of distinct equivalence classes \mbox{$\{ [c] : c\in C\}$} in the proof of \Cref{Prop2.5}\ref{Prop2.5d}$\Rightarrow$\ref{Prop2.5c}, as we can take $\Xi = C$ and $K_c = \{ j\in J : y_j\in [c]\}$.  
    
    Johnson and Odell now complete their proof by setting $Y_\xi = \overline{\operatorname{span}}\, \{y_k : k\in K_\xi\}$ for each $\xi\in\Xi$ (which in our notation translates to $Y_{c} = \overline{\operatorname{span}}\, [c]$ for $c\in C$). This is a separable subspace of~$Y$ because~$K_\xi$ is countable, and we can define an isometric isomorphism $U\colon\bigl(\bigoplus_{\xi\in\Xi}Y_\xi\bigr)_{\ell_p(\Xi)}\to Y$ as follows. For $z = (z_\xi)_{\xi\in\Xi}$, where $z_\xi\in Y_\xi$ for each $\xi\in\Xi$ and $z_\xi =0$ for all but finitely many~$\xi$, set $Uz = \sum_{\xi\in\Xi}z_\xi$, and observe that $U$ is linear with $\lVert Uz\rVert = \bigl(\sum_{\xi\in\Xi}\lVert z_\xi\rVert^p\bigr)^{1/p}$ by~\eqref{Eq:WBJEO}. Therefore~$U$ extends to a linear isometry $\bigl(\bigoplus_{\xi\in\Xi}Y_\xi\bigr)_{\ell_p(\Xi)}\to Y$, which is surjective because its image is closed and dense, as it contains $\bigcup_{\xi\in\Xi}Y_\xi\supseteq\{y_j : j\in J\}$. 
\end{remark}

Considering the generality of \cref{kernelreflexive}, one may ask if a result analogous to \cref{Prop2.5} can be obtained to address the kernel problem for other reflexive transfinite sequence spaces. We propose exploring the following case as an intriguing and more elaborate next step.

\begin{question}
    Let $X$ be a reflexive, transfinite Lorentz sequence space. Is every closed subspace of~$X$ the kernel of a bounded operator $X\to X$? 
\end{question}

\section{Proof of Theorem~\ref{T:main2}}
\noindent Our strategy for resolving the kernel problem for the non-reflexive transfinite sequence spaces~$c_0(\Gamma)$ and~$\ell_1(\Gamma)$ will be to check them against the following simple lemma.

\begin{lemma}\label{quotientsandinjections}
Let~$Y$ be a closed subspace of a Banach space~$X$. Then $Y=\ker T$ for some bounded operator $T\colon X\to X$ if and only if there exists a bounded linear injection $X/Y\to X$.
\end{lemma}

\begin{proof}
Suppose that $Y=\ker T$ for some bounded operator $T\colon X\to X$. Then the fundamental isomorphism theorem tells us that we can define a bounded linear injection $X/Y\to X$ by $x+Y\mapsto Tx$. 

On the other hand, suppose that $R\colon X/Y \to X$ is a bounded linear injection. Then the map 
$T\colon X\to X$, $x\mapsto R(x+Y)$, is a bounded operator with $\ker T = Y$.
\end{proof}

We begin with~$c_0(\Gamma)$. The following notion will be the cornerstone of our approach. 

\begin{definition} A Banach space~$X$ is \emph{weakly compactly generated} (or \emph{WCG} for short) if it contains a weakly compact subset~$W$ such that $\overline{\operatorname{span}}\,W = X$.
\end{definition}

Separable Banach spaces and reflexive Banach spaces are WCG. Furthermore, if there is a bounded operator $X\to Z$ with dense image and~$X$ is WCG, then so is~$Z$; this implies in particular that the class of WCG Banach spaces is closed under taking quotients. Proofs of these easy facts can be found in \cite[p.~575]{BST}, with the remainder of \cite[Chapter~13]{BST} providing a comprehensive introduction to this topic. 

Perhaps the quintessential examples of WCG Banach spaces are the spaces~$c_0(\Gamma)$ for an arbitrary set~$\Gamma$; they are WCG because $W = \{e_\gamma : \gamma \in \Gamma\}\cup\{0\}$ is a weakly compact subset of~$c_0(\Gamma)$ with dense span. Their importance in the study of WCG Banach spaces is due to the following famous theorem of Amir and Lindenstrauss~\cite{AmirLind}.

\begin{theorem}\label{ALWCG}
    Every WCG Banach space admits a bounded linear injection in\-to~$c_0(\Gamma)$ for some set~$\Gamma$.
\end{theorem}

We shall require a slight refinement of it, which allows us to control the cardinality of the set~$\Gamma$ by the \emph{density character} of the Banach space~$X$, that is, the smallest cardinality of a dense subset of~$X$. 

\begin{lemma}\label[lemma]{c0finallemma}
    Let~$X$ be a WCG Banach space with density character~$\Delta$. Then there exists a bounded linear injection $X \to c_0(\Delta)$.
\end{lemma}    
\begin{proof}
  The result is trivial for $X=\{0\}$, so we may suppose that $X$ is non-zero, meaning that~$\Delta$ is an infinite cardinal number. Take a dense subset~$W$ of~$X$ of car\-di\-nal\-i\-ty~$\Delta$, and apply \Cref{ALWCG} to find a bounded linear injection \mbox{$T\colon X \to c_0(\Gamma)$} for some set~$\Gamma$. Since every element of~$c_0(\Gamma)$ has countable support, the subset
  $A = \bigcup_{w \in W} \operatorname{supp} Tw$ of~$\Gamma$ has cardinality at most~$\Delta$. The density of~$W$ in~$X$ implies that $\operatorname{supp}Tx\subseteq A$ for every $x \in X$. Therefore $T[X]$ is contained in~$\overline{\operatorname{span}}\,\{e_\alpha : \alpha \in A\}$, and this subspace embeds isometrically into~$c_0(\Delta)$ because $\lvert A\rvert\le\Delta$.  
\end{proof}

\begin{proof}[Proof of Theorem~{\normalfont{\ref{T:main2}\ref{T:main2i}}}]
Let~$Y$ be a closed subspace of~$c_0(\Gamma)$. The result is trivial if~$\Gamma$ is finite, so we may suppose that~$\Gamma$ is infinite. Then~$c_0(\Gamma)$ is WCG and has density character~$\lvert\Gamma\rvert$, so~$c_0(\Gamma)/Y$ is WCG and has density character $\Delta\le\lvert\Gamma\rvert$.  \Cref{c0finallemma} tells us that there is a bounded linear injection \mbox{$c_0(\Gamma)/Y\to c_0(\Delta)$}. Now the conclusion follows from  \cref{quotientsandinjections} because~$c_0(\Delta)$ embeds isometrically into~$c_0(\Gamma)$.
\end{proof}

Turning our attention to the proof of \Cref{T:main2}\ref{T:main2ii}, our starting point will be the following minor extension of a theorem of Rosenthal~\cite{rose2}. 

\begin{lemma}\label{L:largekernel} Let $T\colon\ell_p(\Gamma)\to\ell_q(\Gamma)$ be a bounded operator, where $1\le q<p<\infty$ and~$\Gamma$ is a set.  Then~$T$ is compact, and~$\Gamma$ contains a countable subset~$A$ such that
\begin{equation}\label{L:largekernel:eq1}
  \overline{\operatorname{span}}\,\{ e_\gamma : \gamma\in\Gamma\setminus A\}\subseteq\ker T. 
\end{equation} 
\end{lemma}

\begin{proof}
Pitt's theorem and Rosenthal's transfinite version of it \cite[Theorem A2]{rose2} tell us that~$T$ is compact. By Schauder's theorem, its adjoint $T^*\colon \ell_{q^*}(\Gamma)\to\ell_{p^*}(\Gamma)$ is also compact, where $1<p^*<q^*\le\infty$ denote the conjugate exponents of~$p$ and~$q$, re\-spec\-tive\-ly. Consequently~$T^*$ has separable image, so we can find a countable subset~$A$ of~$\Gamma$ such that \mbox{$T^*[\ell_{q^*}(\Gamma)]\subseteq\overline{\operatorname{span}}\, \{e_\alpha^* : \alpha\in A\}$}. Hence we have
\[ (Te_\gamma)(\beta) = \langle Te_\gamma, e_\beta^*\rangle = \langle e_\gamma, T^*e_\beta^*\rangle = 0\qquad (\beta\in\Gamma,\,\gamma\in\Gamma\setminus A), \] which proves~\eqref{L:largekernel:eq1}. 
\end{proof}

\begin{corollary}\label{noinjections} Let~$\Gamma$ be an uncountable set, and take $p,q\in[1,\infty)$.
There is a bounded linear injection from~$\ell_p(\Gamma)$ into~$\ell_q(\Gamma)$ if and only if $p\le q$. 
\end{corollary}

\begin{proof} 
The (contrapositive of the) implication $\Rightarrow$ is immediate from \Cref{L:largekernel}; indeed, for 
$1\le q<p<\infty$, \eqref{L:largekernel:eq1} implies that no bounded operator $\ell_p(\Gamma)\to\ell_q(\Gamma)$ is in\-jec\-tive because $\Gamma\setminus A$ is non-empty for every countable subset~$A$ of~$\Gamma$.  

The converse implication is clear because for $1\le p\le q<\infty$, the formal inclusion operator $\ell_p(\Gamma)\to\ell_q(\Gamma)$ is a bounded linear injection of norm~$1$. 
\end{proof}

\begin{proof}[Proof of Theorem~{\normalfont{\ref{T:main2}\ref{T:main2ii}}}]
  Let~$\Gamma$ be an uncountable set, and take $1<p<\infty$. Since $\ell_p(\Gamma)$ has density character~$\lvert\Gamma\rvert$, we can find a bounded linear surjection $S\colon\ell_1(\Gamma)\to \ell_p(\Gamma)$. Then $\ell_1(\Gamma)/\ker S\cong \ell_p(\Gamma)$ by the fundamental isomorphism theorem, so \Cref{noinjections} implies that no bounded linear injection $\ell_1(\Gamma)/\ker S\to \ell_1(\Gamma)$ exists. Hence $\ker S$ is a closed subspace of~$\ell_1(\Gamma)$ which by \cref{quotientsandinjections} is not the kernel of any bounded operator $X\to X$.
  \end{proof}

\section{Wark's Banach space admits a closed subspace that is not a kernel of any bounded operator on it}\label{S:answerforWark}

\noindent
Recall from the introduction that the main result of~\cite{LW} states that the dual~$E_W^*$ of Wark's reflexive, non-separable Banach space~$E_W$ with ``few operators'' admits  a closed subspace that is not the kernel of any bounded operator $E_W^*\to E_W^*$. 
Naturally, this raises the question whether the same holds true for Wark's space~$E_W$ itself. The aim of this section is to verify that it does and, moreover, it  follows easily from  what was already proved in~\cite{LW}.  Before we  justify this claim, let us clarify the precise meaning of the term ``few operators''.  

\begin{definition}\label{D:fewops}
   A non-separable Banach space~$X$ has \emph{few operators} if every bounded operator $X\to X$ is the sum of a scalar multiple of the identity operator and a bounded operator with separable image.
\end{definition}

We can now give the result from~\cite{LW} that we rely on;  it is not stated explicitly therein, but follows by combining \cite[Corollary~2.8]{LW} with the  sentence preceding \cite[Remark~2.5]{LW}, which in the above terminology says: ``the only properties of $E_W$ that we shall make use of are that it is reflexive, non-separable, and [has few operators]''.

\begin{theorem}[Laustsen--White]\label{T:LW}
  Let $X$ be a non-separable, reflexive Banach space which has few operators. Then~$X^*$ contains a closed subspace which is not of the form $\bigcap_{j=1}^n \ker T_j$ for any choice of $n\in\N$ and bounded operators $T_1,\ldots,T_n\colon X^*\to X^*$. 
\end{theorem}

 Our next result uses the term \emph{Asplund space}, which means a Banach space for which every separable subspace has separable dual. 

 \begin{lemma}\label{L:asplund} Let $T\colon X\to Y$ be a bounded operator from a Banach space~$X$ into an Asp\-lund space~$Y$, and suppose that~$T$  has separable image. Then the adjoint operator $T^*\colon Y^*\to X^*$ has separable image as well.
\end{lemma}  

 \begin{proof} Let $Z = \overline{T[X]}$, which is a separable subspace of the Asplund space~$Y$, so its dual~$Z^*$ is also separable. Now let $\widetilde{T}\colon X\to Z$ be the operator~$T$ viewed as an op\-er\-a\-tor into~$Z$, and let $J\colon Z\to Y$ be the inclusion map. Then $T=J\widetilde{T}$, so $T^* = \widetilde{T}^*J^*$, which has separable image because it factors through the separable space~$Z^*$. 
\end{proof}

\begin{corollary}\label{C:fewopduality}
\begin{enumerate}[label={\normalfont{(\roman*)}}]
\item\label{C:fewopduality1} A bounded operator between reflexive Banach spaces has separable image if and only if its adjoint has separable image.    
\item\label{C:fewopduality2} A non-separable, reflexive Banach space has few operators if and only if its dual space has few operators.
\end{enumerate}  
\end{corollary}

\begin{proof} \ref{C:fewopduality1}. This follows by combining \Cref{L:asplund} with the easy observations that reflexive Banach spaces are Asplund and $T^{**}=T$ for every bounded operator~$T$ between reflexive Banach spaces.

\ref{C:fewopduality2}. This is clear from~\ref{C:fewopduality1}.
\end{proof}

\begin{corollary}\label{C:fewopskernel} Let $X$ be a non-separable, reflexive Banach space which has few operators. Then~$X$ contains a closed subspace which is not of the form $\bigcap_{j=1}^n \ker T_j$ for any choice of $n\in\N$ and bounded operators $T_1,\ldots,T_n\colon X\to X$.  
  
\end{corollary}

\begin{proof} This follows by combining \Cref{C:fewopduality}\ref{C:fewopduality2}  with \Cref{T:LW} and the fact that $X^{**}=X$ by reflexivity. \end{proof}

\subsection*{Acknowledgements} This paper was written as part of the first-named author's PhD at Lan\-caster Uni\-ver\-sity. He acknowledges with thanks the funding from the EPSRC (grant number EP/R513076/1) that supported his studies.
 
We are grateful to Bill Johnson for making us aware of the results of~\cite{JohnsonOdell} and to 
the referee for their very careful reading of our paper, and in particular for suggesting \Cref{KeyLemmaB} and its proof to us, which helped us simplify the original proof of \Cref{T:main1}.

\end{document}